\documentclass[amstex,12pt,reqno]{amsart}
\usepackage{amsmath,amsfonts,amssymb,amsthm,enumerate,multicol}%,hyperref
\newtheorem{lemma}{Lemma}
\newtheorem{thm}{Theorem}

\numberwithin{equation}{section}
\begin{document}

\leftline{ \scriptsize \it  }
\title[]
{Global approximation theorems for general Gamma type Operators }
\maketitle

\begin{center}
{\bf Alok Kumar}
\vskip0.2in
Department of Computer Science\\
Dev Sanskriti Vishwavidyalaya Haridwar\\
Haridwar-249411, India\\
\vskip0.2in
%$^2$Abant Izzet Baysal University\\
%Faculty of Science and Arts\\
%Department of Mathematics, 14280 Bolu, Turkey.
%\vskip0.2in

   alokkpma@gmail.com
\vskip0.2in
Dedicated to Prof. P. N. Agrawal
\end{center}

\begin{abstract}
In this paper, we obtained some global approximation results for general Gamma type operators.
\end{abstract}
\section{\textbf{Introduction}}
For a measurable complex valued and locally bounded function defined on $[0,\infty)$, Lupas and M$\ddot{u}$ller \cite{LUP} defined and studied some
approximation properties of linear positive operators $\{G_{n}\}$ defined by
\begin{eqnarray*}
G_{n}(f;x) = \int_{0}^{\infty}g_{n}(x,u)f\bigg(\frac{n}{u}\bigg)du,
\end{eqnarray*}
where
\begin{eqnarray*}
\displaystyle g_{n}(x,u)=\frac{x^{n+1}}{n!}e^{-xu}u^{n}, ~~~~x>0.
\end{eqnarray*}
In \cite{SM}, Mazhar gives an important modifications of the Gamma operators using the same $ g_{n}(x,u)$
\begin{eqnarray*}
 F_{n}(f;x)  &=& \int_{0}^{\infty}\int_{0}^{\infty}g_{n}(x,u)g_{n-1}(u,t)f(t)dudt\\
 &=&\frac{(2n)! x^{n+1}}{n!(n-1)!}\int_{0}^{\infty}\frac{t^{n-1}}{(x+t)^{2n+1}}f(t)dt,\,\, n>1, \,\, x>0.
\end{eqnarray*}

Recently, Karsli \cite{HK} considered the following Gamma type linear and positive operators
\begin{eqnarray*}
 L_{n}(f;x)  &=& \int_{0}^{\infty}\int_{0}^{\infty}g_{n+2}(x,u)g_{n}(u,t)f(t)dudt\\
 &=&\frac{(2n+3)!x^{n+3}}{n!(n+2)!}\int_{0}^{\infty}\frac{t^{n}}{(x+t)^{2n+4}}f(t)dt, \,\, x>0,
\end{eqnarray*}
and obtained some approximation results.\\
In \cite{HMA}, Karsli and $\ddot{O}$zarslan obtained some local and global approximation results for the operators $L_{n}(f;x)$.\\
In 2007, Mao \cite{MO} define the following generalised Gamma type linear and positive operators
\begin{eqnarray*}
 M_{n,k}(f;x)  &=& \int_{0}^{\infty}\int_{0}^{\infty}g_{n}(x,u)g_{n-k}(u,t)f(t)dudt\\
 &=&\frac{(2n-k+1)!x^{n+1}}{n!(n-k)!}\int_{0}^{\infty}\frac{t^{n-k}}{(x+t)^{2n-k+2}}f(t)dt, \,\, x>0,
\end{eqnarray*}
which includes the operators $F_{n}(f;x)$ for $k=1$ and $ L_{n-2}(f;x)$ for $k=2$.\\
 Some approximation properties of $M_{n,k}$ were studied in  \cite{HKL} and \cite{HPM}. Several authors obtain the global approximation results for different operators
 (see \cite{MB}, \cite{MF} and \cite{ZF}).\\
  We can rewrite the operators $M_{n,k}(f;x)$ as
\begin{eqnarray}\label{eq11}
 M_{n,k}(f;x) = \int_{0}^{\infty}K_{n,k}(x,t)f(t)dt,
\end{eqnarray}
where
\begin{eqnarray*}
K_{n,k}(x,t)=\frac{(2n-k+1)!x^{n+1}}{n!(n-k)!}\frac{t^{n-k}}{(x+t)^{2n-k+2}}, \,\, ~~x,t\in(0,\infty).
\end{eqnarray*}

In this paper, we study some global approximation results of the operators $ M_{n,k}$. Let $p\in N_{0}$(set of non-negative integers), $f\in C_{p},$  where $C_{p}$ is a polynomial weighted space with the weight function $w_{p}$,
\begin{eqnarray}\label{eq12}
w_{0}(x)=1,\,\,\, w_{p}(x)=\frac{1}{1+x^{p}}, \,\,\, p\geq1,
\end{eqnarray}
and $C_{p}$ is the set of all real valued functions $f$ for which $w_{p}f$ is bounded and uniformly continuous on $[0,\infty)$.\\
The norm on $C_{p}$ is defined by
\begin{eqnarray*}
||f||_{p}=\sup_{x\in[0,\infty)}w_{p}(x)|f(x)|,  \,\, f\in C_{p}[0,\infty).
\end{eqnarray*}

 We also consider the following Lipschitz classes:
\begin{eqnarray*}
\omega_{p}^{2}(f;\delta)=\sup_{h\in(0,\delta]}||\Delta_{h}^{2}f||_{p},
\end{eqnarray*}
\begin{eqnarray*}
\Delta_{h}^{2}f(x)=f(x+2h)-2f(x+h)+f(x),
\end{eqnarray*}
\begin{eqnarray*}
\omega_{p}^{1}(f;\delta)=\sup\{w_{p}(x)|f(t)-f(x)|:|t-x|\leq\delta \, and \,\, t,x\geq0\},
\end{eqnarray*}
\begin{eqnarray*}
Lip_{p}^{2}\alpha=\{f\in C_{p}[0,\infty):\omega_{p}^{2}(f;\delta)=O(\delta^{\alpha})\, as \, \delta\rightarrow 0^{+}\},
\end{eqnarray*}
where $h>0$ and $\alpha\in(0,2]$.

\section{\textbf{Auxiliary Results}}
In this section we give some preliminary results which will be used in the proofs of our main theorems.\\
Let us consider
\begin{eqnarray*}
e_{m}(t)=t^{m}, \,\,\, \varphi_{x,m}(t)=(t-x)^{m}, \,\, m\in N_{0}.
\end{eqnarray*}
\begin{lemma}\label{l1}\cite{HKL}
For any $m\in N_{0}$(set of non-negative integers), $m\leq n-k$
\begin{eqnarray}\label{eq13}
 M_{n,k}(t^{m};x)=\frac{[n-k+m]_{m}}{[n]_{m}}x^{m}
 \end{eqnarray}
where $n, k\in N$ and $[x]_{m}=x(x-1)...(x-m+1), [x]_{0}=1, x\in R.$ \\
In particular for m = 0, 1, 2... in (2.1)  we get
\begin{enumerate}[(i)]
   \item $ M_{n,k}(1;x)= 1,$
   \item $ M_{n,k}(t;x)=\displaystyle\frac{n-k+1}{n}x,$
   \item $ M_{n,k}(t^{2};x)=\displaystyle\frac{(n-k+2)(n-k+1)}{n(n-1)}x^{2}$.
\end{enumerate}
\end{lemma}
\begin{lemma}\label{l2}\cite{HKL}
Let $m\in N_{0}$ and fixed $x\in(0,\infty)$, then
\begin{eqnarray*}
M_{n,k}(\varphi_{x,m};x)=\left(\sum_{j=0}^{m}(-1)^{j}{m\choose j}\frac{(n-m+j)!(n-k+m-j)!}{n!(n-k)!}\right)x^{m}.
\end{eqnarray*}
\end{lemma}
\begin{lemma}\label{l3} For $m=0,1,2,3,4$, one has
\begin{enumerate}[(i)]
   \item $ M_{n,k}(\varphi_{x,0};x)=\displaystyle1,$
   \item $ M_{n,k}(\varphi_{x,1};x)=\displaystyle\frac{1-k}{n}x,$
   \item $ M_{n,k}(\varphi_{x,2};x)=\displaystyle\frac{k^{2}-5k+2n+4}{n(n-1)}x^{2},$
   \item $ M_{n,k}(\varphi_{x,3};x)=\displaystyle\frac{-k^{3}+12k^{2}-17k+n(18-12k)+24}{n(n-1)(n-2)}x^{3},$
   \item $ M_{n,k}(\varphi_{x,4};x)=\displaystyle\frac{k^{4}-22k^{3}+k^{2}(143+12n)-k(314+108n)+12n^{2}+268n+192}{n(n-1)(n-2)(n-3)}x^{4},$
   \item $ M_{n,k}(\varphi_{x,m};x)=\displaystyle O\left(n^{-[(m+1)/2]}\right)$.
\end{enumerate}
\end{lemma}
\begin{proof}
Using Lemma \ref{l2}, we get Lemma \ref{l3}.
\end{proof}

\begin{thm} \label{t1}
For the operators $M_{n,k}$ and for fixed $p\in N_{0}$, there exists a positive constant $N_{p,k}$ such that
\begin{eqnarray}\label{eq14}
w_{p}(x)M_{n,k}\left(\frac{1}{w_{p}};x\right)\leq N_{p,k}.
\end{eqnarray}
Furthermore, for all $f\in C_{p}[0,\infty),$ we have
\begin{eqnarray}\label{eq15}
\|M_{n,k}(f;.)\|_{p}\leq N_{p,k}\|f\|_{p},
\end{eqnarray}
which guarantees that $M_{n,k}$ maps $C_{p}[0,\infty)$ into $C_{p}[0,\infty)$.
\end{thm}
\begin{proof}
For $p=0$, (\ref{eq14}) follows immediately. Using Lemma \ref{l1}, we get
\begin{eqnarray*}
w_{p}(x)M_{n,k}\left(\frac{1}{w_{p}};x\right)&=&w_{p}(x)\left(M_{n,k}(e_{0};x)+M_{n,k}(e_{p};x)\right)\\
&=&w_{p}(x)\left(1+\frac{(n-p)!(n-k+p)!}{n!(n-k)!}x^{p}\right)\\
&\leq& N_{p,k}w_{p}(x)(1+x^{p})=N_{p,k},
\end{eqnarray*}
where
\begin{eqnarray*}
N_{p,k}=\max\left\{\sup_{n}\frac{(n-p)!(n-k+p)!}{n!(n-k)!},1\right\}.
\end{eqnarray*}
Observe that for all $f\in C_{p}$ and every $x\in(0,\infty)$, we get
\begin{eqnarray*}
w_{p}(x)\left|M_{n,k}(f;x)\right|&\leq&w_{p}(x)\frac{(2n-k+1)!x^{n+1}}{n!(n-k)!}\int_{0}^{\infty}\frac{t^{n-k}}{(x+t)^{2n-k+2}}|f(t)| \frac{w_{p}(t)}{w_{p}(t)}dt\\
&\leq&\|f\|_{p}w_{p}(x)M_{n,k}\left(\frac{1}{w_{p}};x\right)\\
&\leq& N_{p,k}\|f\|_{p}.
\end{eqnarray*}
Taking supremum over $x\in(0,\infty),$ we get (\ref{eq15}).
\end{proof}

\begin{lemma}\label{l4} For the operators $M_{n,k}$ and fixed $p\in N_{0}$, there exists a positive constant $N_{p,k}$ such that
\begin{eqnarray*}
w_{p}(x)M_{n,k}\left(\frac{\varphi_{x,2}}{w_{p}(t)};x\right)\leq N_{p,k}\frac{x^{2}}{n}.
\end{eqnarray*}
\end{lemma}
\begin{proof}
Using Lemma (\ref{l3}), we can write
\begin{eqnarray*}
w_{0}(x)M_{n,k}\left(\frac{\varphi_{x,2}}{w_{0}(t)};x\right)&=&\displaystyle\frac{k^{2}-5k+2n+4}{n(n-1)}x^{2}\\
&\leq&N_{p,k}\frac{x^{2}}{n},
\end{eqnarray*}
which gives the result for $p=0$.\\
Let $p\geq1$. Then using Lemma \ref{l1} and Lemma \ref{l3}, we get
\begin{eqnarray*}
M_{n,k}\left(\frac{\varphi_{x,2}}{w_{p}(t)};x\right)&=&M_{n,k}(e_{p+2};x)-2xM_{n,k}(e_{p+1};x)+x^{2}M_{n,k}(e_{p};x)+M_{n,k}(\varphi_{x,2};x)\\
&=&\frac{(n-p-2)!(n-k+p+2)!}{n!(n-k)!}x^{p+2}-2\frac{(n-p-1)!(n-k+p+1)!}{n!(n-k)!}x^{p+2}\\
&+&\frac{(n-p)!(n-k+p)!}{n!(n-k)!}x^{p+2}+\frac{k^{2}-5k+2n+4}{n(n-1)}x^{2}\\
&\leq& N_{p,k}\frac{x^{2}}{n}(1+x^{p}),
\end{eqnarray*}
where $N_{p,k}$ is a positive constant. Hence, the proof is completed.
\end{proof}

\section{\textbf{Rate of Convergence}}
 Let $p\in N_{0}$. By $C_{p}^{2}[0,\infty),$ we denote the space of all functions $f\in C_{p}[0,\infty)$ such that $f',f''\in C_{p}[0,\infty)$.\\

\begin{thm} \label{t2}
Let $p\in N_{0}$, $n\in N$ and $g\in C_{p}^{1}[0,\infty)$, there exists a positive constant $N_{p,k}$ such that
\begin{eqnarray*}
w_{p}(x)|M_{n,k}(f;x)-f(x)|\leq N_{p,k}\|f'\|_{p}\frac{x}{\sqrt{n}}
\end{eqnarray*}
for all $x\in(0,\infty)$.
\end{thm}
\begin{proof}
We have
\begin{eqnarray*}
f(t)-f(x)=\int_{x}^{t}f'(v)dv.
\end{eqnarray*}
By using linearity of $M_{n,k}$ we get
\begin{eqnarray}\label{eq16}
M_{n,k}(f;x)-f(x)=M_{n,k}\left(\int_{x}^{t}f'(v)dv;x\right).
\end{eqnarray}
Remark that
\begin{eqnarray*}
\left|\int_{x}^{t}f'(v)dv\right|\leq\|f'\|_{p}\left|\int_{x}^{t}\frac{dv}{w_{p}(v)}\right|\leq\|f'\|_{p}|t-x|\left(\frac{1}{w_{p}(t)}+\frac{1}{w_{p}(x)}\right).
\end{eqnarray*}
From (\ref{eq16}) we obtain
\begin{eqnarray*}
w_{p}(x)|M_{n,k}(f;x)-f(x)|\leq\|f'\|_{p}\left\{M_{n,k}(|\varphi_{x,1}|;x)+w_{p}(x)M_{n,k}\left(\frac{|\varphi_{x,1}|}{w_{p}(t)};x\right)\right\}.
\end{eqnarray*}
Using Cauchy-Schwarz inequality, we can write
\begin{eqnarray*}
M_{n,k}(|\varphi_{x,1}|;x)\leq \left(M_{n,k}(|\varphi_{x,2}|;x)\right)^{1/2},
\end{eqnarray*}
\begin{eqnarray*}
M_{n,k}\left(\frac{|\varphi_{x,1}|}{w_{p}(t)};x\right)\leq\left(M_{n,k}\left(\frac{1}{w_{p}(t)};x\right)\right)^{1/2}\left(M_{n,k}\left(\frac{\varphi_{x,2}}{w_{p}(t)};x\right)\right)^{1/2}.
\end{eqnarray*}
Using Leema \ref{l3}, Theorem \ref{t1} and Lemma \ref{l4}, we obtain
\begin{eqnarray*}
w_{p}(x)|M_{n,k}(f;x)-f(x)|\leq N_{p,k}\|f'\|_{p}\frac{x}{\sqrt{n}}.
\end{eqnarray*}
\end{proof}
\begin{lemma}\label{l5} Let $p\in N_{0}$, If
\begin{eqnarray}\label{eq18}
H_{n,k}(f;x)=M_{n,k}(f;x)-f\left(x+\frac{1-k}{n}x\right)+f(x),
\end{eqnarray}
then there exists a positive constant $N_{p,k}$ such that for all $x\in(0,\infty)$ and $n\in N$, we have
\begin{eqnarray*}
w_{p}(x)|H_{n,k}(g;x)-g(x)|\leq N_{p,k}||g''||_{p}\frac{x^{2}}{n}
\end{eqnarray*}
for any function $g\in C_{p}^{2}$.
\end{lemma}
\begin{proof}
From Lemma \ref{l1}, we observe that the operators $H_{n,k}$ are linear and reproduce the linear functions.\\
Hence
\begin{eqnarray*}
H_{n,k}(\varphi_{x,1};x)=0.
\end{eqnarray*}
Let $g\in C_{p}^{2}$. By the Taylor formula one can write
\begin{eqnarray*}
g(t)-g(x)=(t-x)g'(x)+\int_{x}^{t}(t-v)g''(v)dv, \,\,\, t\in(0,\infty).
\end{eqnarray*}
Then,\\
\noindent
$|H_{n,k}(g;x)-g(x)|$
\begin{eqnarray*}
&=&|H_{n,k}(g(t)-g(x));x|=\left|H_{n,k}\left(\int_{x}^{t}(t-v)g''(v)dv;x\right)\right|\\
&=&\left|M_{n,k}\left(\int_{x}^{t}(t-v)g''(v)dv;x\right)-\int_{x}^{x+\frac{1-k}{n}x}\left(x+\frac{1-k}{n}x-v\right)g''(v)dv\right|.
\end{eqnarray*}
Since
\begin{eqnarray*}
\left|\int_{x}^{t}(t-v)g''(v)dv\right|\leq \frac{\parallel g''\parallel_{p}(t-x)^{2}}{2}\left(\frac{1}{w_{p}(x)}+\frac{1}{w_{p}(t)}\right)
\end{eqnarray*}
and
\begin{eqnarray*}
\left|\int_{x}^{x+\frac{1-k}{n}x}\left(x+\frac{1-k}{n}x-v\right)g''(v)dv\right|\leq\frac{||g''||_{p}}{2w_{p}(x)}\left(\frac{1-k}{n}x\right)^{2},
\end{eqnarray*}
we get
\begin{eqnarray*}
w_{p}(x)|H_{n,k}(g;x)-g(x)|\leq\frac{\|g''\|_{p}}{2}\left[M_{n,k}(\varphi_{x,2};x)+w_{p}(x)M_{n,k}\left(\frac{\varphi_{x,2}}{w_{p}(t)};x\right)\right]+\frac{\|g''\|_{p}}{2}\left(\frac{1-k}{n}x\right)^{2}.
\end{eqnarray*}
Hence by Lemma \ref{l4}, we obtain
\begin{eqnarray*}
w_{p}(x)|H_{n,k}(g;x)-g(x)|\leq N_{p,k}\|g''\|_{p}\frac{x^{2}}{n}
\end{eqnarray*}
for any function $g\in C_{p}^{2}$. The Lemma is proved.
\end{proof}

 The next theorem is the main result of this section.
\begin{thm} \label{t4}
Let $p\in N_{0}$, $n\in N$ and $f\in C_{p}[0,\infty)$, then there exists a positive constant $N_{p,k}$ such that
\begin{eqnarray*}
w_{p}(x)\left|M_{n,k}(f;x)-f(x)\right|\leq N_{p,k}\omega_{p}^{2}\left(f,\frac{x}{\sqrt{n}}\right)+\omega_{p}^{1}\left(f,\frac{1-k}{n}x\right).
\end{eqnarray*}
 Furthermore, if $f\in Lip_{p}^{2}\alpha$ for some $\alpha\in(0,2],$ then
\begin{eqnarray*}
w_{p}(x)\left|M_{n,k}(f;x)-f(x)\right|\leq N_{p,k}\left(\frac{x^{2}}{n}\right)^{\alpha/2}+\omega_{p}^{1}\left(f,\frac{1-k}{n}x\right),
\end{eqnarray*}
holds.
\end{thm}
\begin{proof}
Let $p\in N_{0}$, $f\in C_{p}[0,\infty)$ and $x\in(0,\infty)$ be fixed. We consider the Steklov means  of $f$ by $f_{h}$ and given by the formula
\begin{eqnarray*}
f_{h}(x)=\frac{4}{h^{2}}\int_{0}^{h/2}\int_{0}^{h/2}\{2f(x+s+t)-f(x+2(s+t))\}ds dt,
\end{eqnarray*}
for $h>0$ and  $x\geq0$. We have
\begin{eqnarray*}
f(x)-f_{h}(x)=\frac{4}{h^{2}}\int_{0}^{h/2}\int_{0}^{h/2}\Delta_{s+t}^{2}f(x)ds dt,
\end{eqnarray*}
which gives
\begin{eqnarray}\label{eq19}
\|f-f_{h}\|_{p}\leq \omega_{p}^{2}(f,h).
\end{eqnarray}
Furthermore, we have
\begin{eqnarray*}
f_{h}^{''}(x)=\frac{1}{h^{2}}\left(8\Delta_{h/2}^{2}f(x)-\Delta_{h}^{2}f(x)\right),
\end{eqnarray*}
and
\begin{eqnarray}\label{eq20}
\|f_{h}^{''}\|_{p}\leq \frac{9}{h^{2}}\omega_{p}^{2}(f,h).
\end{eqnarray}
From (\ref{eq19}) and (\ref{eq20}) we conclude that $f_{h}\in C_{p}^{2}[0,\infty)$ if $f\in C_{p}[0,\infty)$.\\
Moreover
\begin{eqnarray*}
|M_{n,k}(f;x)-f(x)|&\leq&H_{n,k}(|f(t)-f_{h}(t);x|)+|f(x)-f_{h}(x)|\\
&+&|H_{n,k}(f_{h};x)-f_{h}(x)|+\left|f\left(x+\frac{1-k}{n}x\right)-f(x)\right|,
\end{eqnarray*}
where $H_{n,k}$ is defined in (\ref{eq18}).\\
 Since $f_{h}\in C_{p}^{2}[0,\infty)$ by the above, it follows from Theorem \ref{t1} and Lemma \ref{l5} that\\
\begin{eqnarray*}
w_{p}(x)\left|M_{n,k}(f;x)-f(x)\right|&\leq&(N+1)\|f-f_{h}\|_{p}+N_{p,k}\|f_{h}^{''}\|_{p}\frac{x^{2}}{n}\\
&+&w_{p}(x)\left|f\left(x+\frac{1-k}{n}x\right)-f(x)\right|.
\end{eqnarray*}
By (\ref{eq19}) and (\ref{eq20}), the last inequality yields that
\begin{eqnarray*}
w_{p}(x)\left|M_{n,k}(f;x)-f(x)\right|\leq N_{p,k}\omega_{p}^{2}(f;h)\left(1+\frac{1}{h^{2}}\frac{x^{2}}{n}\right)+\omega_{p}^{1}\left(f,\frac{1-k}{n}x\right).
\end{eqnarray*}
Thus, choosing $h=\frac{x}{\sqrt{n}}$, the first part of the proof is completed.\\
The remainder of the proof can be easily obtained from the definition of the space $Lip_{p}^{2}\alpha$.
\end{proof}

{\bf Acknowledgements}
The author is extremely grateful to the referee for making valuable suggestions leading to the overall improvements in the paper.


\begin{thebibliography}{99}

\bibitem{MB} Becker M. Global approximation theorems for Sz$\acute{a}$sz-Mirakyan and Baskakov operators in polynomial weight spaces. Indiana University Mathematics Journal 1978; 27 (1): 127-142.

\bibitem{RAD} DeVore RA, Lorentz GG. Constructive Approximation. Berlin: Springer, 1993.

\bibitem{MF} Felten M. Local and global approximation theorems for positive linear operators. J. Approx. Theory 1998; 94: 396-419.

\bibitem{ZF} Finta Z. Direct local and global approximation theorems for some linear positive operators. Analysis in Theory and Applications  2004; 20 (4): 307-322.

\bibitem{AIZ} $\dot{I}$zgi A. Voronovskaya type asymptotic approximation by modified gamma operators. Appl. Math. Comput. 2011; 217: 8061-8067.

\bibitem{ABI} $\dot{I}$zgi  A, B$\ddot{u}$y$\ddot{u}$kyazici I. Approximation and rate of approximation on unbounded intervals. Kastamonu Edu. J. Okt. 2003; 11: 451-460(in Turkish).

\bibitem{HK} Karsli H. Rate of convergence of a new Gamma type operators for the functions with derivatives of bounded variation. Math. Comput. Modell.  2007; 45 (5-6): 617-624.

\bibitem{HKL} Karsli H. On convergence of general Gamma type operators. Anal. Theory Appl. 2011; Vol. 27, No.3: 288-300.

\bibitem{HPM} Karsli H, Agrawal PN, Goyal M. General Gamma type operators based on q-integers. Appl. Math. Comput.  2015; 251: 564-575.

\bibitem{HVI} Karsli H, Gupta V, Izgi A. Rate of pointwise convergence of a new kind of gamma operators for functions of bounded variation. Appl. Math. Letters 2009; 22: 505-510.

\bibitem{HMA} Karsli H, $\ddot{O}$zarslan MA. Direct local and global approximation results for operators of gamma type. Hacet. J. Math. Stat. 2010; 39: 241-253.

\bibitem{LUP} Lupas A, M$\ddot{u}$ller M. Approximationseigenschaften der Gammaoperatören. Mathematische Zeitschrift 1967; 98: 208-226.

\bibitem{SM} Mazhar SM. Approximation by positive operators on infinite intervals. Math. Balkanica 1991; 5 (2): 99-104.

\bibitem{MO} Mao LC. Rate of convergence of Gamma type operator. J. Shangqiu Teachers Coll. 2007; 12: 49-52.

\bibitem{LPK} Rempulska L, Tomczak K. Approximation by certain linear operators preserving $x^2$. Turkish J. Math. 2009; 33: 273-281.





\end{thebibliography}
\end{document}